\documentclass[11pt]{amsart}
\usepackage{amsmath,amsthm,amsfonts,amssymb,bm,graphicx,color,comment}

\usepackage
{hyperref}
\hypersetup{colorlinks=true,citecolor=blue,linkcolor=blue,urlcolor=blue}

\theoremstyle{plain}
\newtheorem{theorem}{Theorem}
\newtheorem{lemma}[theorem]{Lemma}

\numberwithin{theorem}{section}

\newcommand{\lz}{\left(}
\newcommand{\pz}{\right)}

\renewcommand{\epsilon}{\varepsilon}
\newcommand{\C}{\mathbb{C}}
\newcommand{\Z}{\mathbb{Z}}

\newcommand{\R}{\mathbb{R}}

\newcommand{\lab}{\left|}
\newcommand{\rab}{\right|}
\newcommand{\unitint}{1_{[0,1]}}

\newcommand{\re}{\mathrm{Re}}

\newcommand{\im}{\mathrm{Im}}

\newcommand{\bfrac}[2]{\lz\frac{#1}{#2}\pz}

\newcommand{\msection}[1]{\section{\large \bf #1}}
\newcommand{\sw}{S(X,Y;\phi,\psi)}
\newcommand{\s}{S(X,Y)}
\newcommand{\res}{\mathrm{res}}

\renewcommand{\mod}[1]{\text{ (mod $#1$)}}
\newcommand{\odd}{\mathrm{\ odd}}

\renewcommand{\phi}{\varphi}

\newcommand{\leg}[2]{\left(\frac{#1}{#2}\right)}

\usepackage[utf8]{inputenc}
\begin{document}
	
	\title[Mean value of real characters]{Mean value of real characters using a double Dirichlet series} 
	
	\author{M. \v Cech}
	
	\address[Martin \v Cech]
	{Concordia University, Montreal, Canada}
	\email{martin.cech@concordia.ca}

\begin{abstract} We study the double character sum $\sum\limits_{\substack{m\leq X,\\m\odd}}\sum\limits_{\substack{n\leq Y,\\n\odd}}\leg mn$ and its smoothly weighted counterpart. An asymptotic formula with power saving error term was obtained by Conrey, Farmer and Soundararajan by applying the Poisson summation formula. The result is interesting, because the main term involves a non-smooth function. In this paper, we apply the inverse Mellin transform twice and study the resulting double integral that involves a double Dirichlet series. This method has two advantages -- it leads to a better error term, and the surprising main term naturally arises from three residues of the double Dirichlet series.
\end{abstract}	

\maketitle

\msection{Introduction}
We study the double character sum\begin{equation}\label{DoubleSum}
\s:=\sum_{\substack{m\leq X,\\m\odd}}\sum_{\substack{n\leq Y,\\n\odd}}\leg mn.
\end{equation}
This sum was studied by Conrey, Farmer and Soundararajan in \cite{CFS}, where the authors give an asymptotic formula valid for all large $X$ and $Y$.

If $Y=o(X/\log X)$, then the main term of $S(X,Y)$ comes from the terms where $n$ is a square, and the error term can be estimated using the P\'olya-Vinogradov inequality. In particular, we get that in this range, \begin{equation}\label{PolyaVinogradov}
\s=\frac{2}{\pi^2}XY^{1/2}+O(Y^{3/2}\log Y+Y^{1/2+\epsilon}+X\log Y),
\end{equation}
and similarly for $X=o(Y/\log Y)$.

Conrey, Farmer and Soundararajan showed that there is a transition in the behavior of $S(X,Y)$ when $X,Y$ are of similar size. In particular, they proved the following asymptotic formula, which is valid for all large $X,Y$: \begin{equation}\label{AsymptoticFromCFS}
S(X,Y)=\frac2{\pi^2}X^{3/2}C\lz\frac YX\pz+O\lz\lz XY^{7/16}+YX^{7/16}\pz\log(XY)\pz,
\end{equation} where \begin{equation}\label{DefinitionOfC}
C(\alpha)=\alpha+\alpha^{3/2}\frac2{\pi}\sum_{k=1}^\infty\frac1{k^2}\int_{0}^{1/\alpha}\sqrt y\sin\bfrac{\pi k^2}{2y} dy.
\end{equation} The size of the main term in this formula is $XY^{1/2}+YX^{1/2}$, so it is always larger than the error term. The result is interesting, because $C(\alpha)$ is a non-smooth function. For a heuristic explanation why such functions arise in this type of problems, see the first section in \cite{Pet} and the references therein. 

Conrey, Farmer and Soundararajan also gave the following asymptotic estimates for $C(\alpha)$: \begin{equation}\label{AsymptoticEstimateForC1}
C(\alpha)=\sqrt\alpha+\frac{\pi}{18}\alpha^{\frac32}+O\lz\alpha^{5/2}\pz\hbox{ \ as $\alpha\rightarrow0$},
\end{equation}
and \begin{equation}\label{AsymptoticEstimateForC2}
C(\alpha)=\alpha+O\lz\alpha^{-1}\pz\hbox{ \ as $\alpha\rightarrow\infty$}.
\end{equation}

To prove \eqref{AsymptoticFromCFS}, Conrey, Farmer and Soundararajan applied the Poisson summation formula and estimated the sums of Gauss sums which appeared in the computation. Similar techniques were used in the work of Gao and Zhao to compute the mean value in other families of characters, such as cubic and quartic Dirichlet characters \cite{GZ2}, and some quadratic, cubic and quartic Hecke characters \cite{GZ1}. Gao used similar methods to compute the mean value of the divisor function twisted by quadratic characters \cite{Gao}.
\smallskip

Our approach is to rewrite $S(X,Y)$ as a double integral by using the inverse Mellin transform twice. The integral will then involve the double Dirichlet series $$A(s,w)=\sum_{m\odd}\sum_{n\odd}\frac{\leg mn}{m^wn^s},$$ which was studied by Blomer \cite{Blo}, who showed that it admits a meromorphic continuation to the whole $\C^2$ and determined the polar lines. We then shift the integrals to the left and compute the contribution of the residues. The quality of the error term depends whether we assume the truth of the Riemann Hypothesis because the zeros of $\zeta(s)$ appear in the location of the poles of $A(s,w)$, and also in the contribution of the residues. 

An interesting feature of our proof is that the 3 polar lines from which our main term arises naturally correspond to the contribution of squares (the polar lines $s=1$ and $w=1$), and the transition term where the non-smooth function appears (the polar line $s+w=3/2$).

A more general theory of multiple Dirichlet series has been developed by Bump, Chinta, Diaconu, Friedberg, Goldfeld, Hoffstein and others. We refer the reader interested in the theory and its applications to the expository articles \cite{Bum}, \cite{BFH}, \cite{CFH}, the paper \cite{DGH} or the book \cite{BFG}.

\smallskip

To state our results, we first define the smooth sum \begin{equation}\label{DefinitionOfSmoothSum}
S(X,Y;\phi,\psi)=\sum_{m,n\odd}\leg mn \phi(m/X)\psi(n/Y),
\end{equation} where $\phi,\psi$ are nonnegative smooth functions supported in $(0,1)$.

If we denote by $\hat f$ the Mellin transform of $f$ (see \eqref{DefinitionMellinTransform}), the main result is the following:
\begin{theorem}\label{MainTheoremSmoothed}
	Let $\epsilon>0$. Then for all large $X,Y,$ we have \begin{equation}\label{MainResultSmoothed}
	\sw=\frac{2}{\pi^2}\cdot X^{3/2}\cdot D\lz\frac YX;\phi,\psi\pz+O_\epsilon(XY^\delta+YX^\delta),
	\end{equation}
	where $\delta=\epsilon$, and $$\begin{aligned}
	D(\alpha;\phi,\psi)=&\frac{\hat\phi(1)\hat\psi\bfrac12\alpha^{1/2}+\hat\psi(1)\hat\phi\bfrac12\alpha}{2}+\\[5pt]&\hspace{-50pt}+\frac1{i\sqrt\pi }\int\limits_{(3/4)}\bfrac{\alpha}{2\pi}^s\cdot\hat\phi\lz\frac32-s\pz\hat\psi(s)\Gamma\lz s-\frac12\pz\sin\bfrac{\pi s}2\zeta(2s-1)ds.
	\end{aligned}$$ If we assume the Riemann Hypothesis, then we can take $\delta=-1/4+\epsilon$.
\end{theorem}
We can remove the smooth weights and obtain the following asymptotic formula for $S(X,Y)$, which improves the error term in \eqref{AsymptoticFromCFS}:
\begin{theorem}\label{MainTheorem}
	Let $\epsilon>0$. Then for all large $X,Y$, we have \begin{equation}\label{MainResult}
	S(X,Y)=\frac2{\pi^2}\cdot X^{3/2}\cdot D\bfrac YX + O_\epsilon(XY^{1/4+\epsilon}+YX^{1/4+\epsilon}),
	\end{equation}
	where \begin{equation}\label{D(alpha)}
	D(\alpha)=\sqrt\alpha+\alpha-\frac1{i\sqrt\pi }\int\limits_{(3/4)}\bfrac{\alpha}{2\pi}^s\cdot\frac{\Gamma\lz s-\frac32\pz\sin\bfrac{\pi s}2\zeta(2s-1)}{s}ds.
	\end{equation}
\end{theorem}

We show in Section \ref{SectionComparingMainTerms} that $D(\alpha)=C(\alpha)$, so our main term agrees with that of Conrey, Farmer and Soundararajan.

Let us also remark that a similar asymptotic can be obtained if the integers $m,n$ were restricted to lie in a congruence class modulo $8$ by working with a suitable combination of the twisted double Dirichlet series, as defined in \eqref{ZetWithCharacters}.

\section*{\textbf{Acknowledgments}}

I would like to thank my supervisor Chantal David for her help and valuable comments, and to David Farmer for comments on an earlier version of this paper and suggestion of the proof in Section \ref{SectionComparingMainTerms}. I would also like to thank the anonymous referee for many useful comments and suggestions.

\msection{Preliminaries and notation}
Throughout the paper, $\epsilon$ will denote a sufficiently small positive number, different at each occurrence, and all implied constants are allowed to depend on $\epsilon$.

We follow the notation of \cite{Blo}. For integers $m,n$, we denote by $\chi_m(n)$ the Kronecker symbol $$\chi_m(n)=\leg mn.$$ Assume that $m$ is odd and write it as $m=m_0m_1^2$ with $m_0$ squarefree. Then $\chi_m$ is a character of conductor $|m_0|$ if $m\equiv1\mod 4$ and $|4m_0|$ if $m\equiv3\mod 4.$ We denote by $\psi_1,\psi_{-1},\psi_2,\psi_{-2}$ the four Dirichlet characters modulo 8 given by the Kronecker symbol $\psi_j(n)=\leg jn.$ We also let $$\tilde\chi_m=\begin{cases}\chi_m,&\hbox{\text{ if $m\equiv1\mod 4$,}}\\\chi_{-m},&\hbox{\text{ if $m\equiv3\mod 4$.}}\end{cases}$$ With this notation, quadratic reciprocity tells us that for odd positive integers $m,n$, \begin{equation}\label{QuadraticReciprocity}
\chi_m(n)=\tilde\chi_n(m).
\end{equation}

The fundamental discriminants $m$ correspond to primitive real characters of conductor $|m|$. In such cases, the completed L-function is $$\Lambda(s,\chi_{m})=\bfrac {|m|}\pi ^{\frac{s+a}2}\Gamma\bfrac{s+a}2L(s,\chi_{m}),$$ where $a=0$ or $1$ depending on whether the character is even or odd, i.e., whether $\chi_m(-1)=1$ or $-1$, and we have the functional equation \begin{equation}\label{FunctionalEquationDirichletLFunction}
\Lambda(s,\chi_m)=\Lambda(1-s,\chi_m).
\end{equation} 
All primitive real characters can be uniquely written as $\chi_{m_0}\psi_j$ for some positive odd squarefree integer $m_0$ and $j\in\{\pm1,\pm2\}$.

If $m$ is not a fundamental discriminant, then $\chi_m$ is a character of conductor $m_0\mid4m$, and we have \begin{equation}\label{LFunctionNonPrimitive}
L(s,\chi_m)=L(s,\chi_{m_0})\cdot\prod_{p\mid\frac{|m|}{m_0}}\lz1-\frac{\chi_{m_0}(p)}{p^s}\pz.
\end{equation}

A subscript $2$ of an L-function means that the Euler factor at $2$ is removed, so in particular \begin{equation}\label{RemovedEulerFactor}
L_2(s,\chi)=\sum_{n\odd}\frac{\chi(n)}{n^s}.
\end{equation}

We now record two estimates that will be used later. 

The first estimate holds for any $s$ with $\re(s)\geq1/2$:\begin{equation}\label{LindelofOnAverage}
\sum_{\substack{m\leq X,\\m\odd}}|L_2(s,\chi_m\psi_j)|\ll_{\epsilon} X^{1+\epsilon}|s|^{\frac14+\epsilon}.
\end{equation}
It follows after applying Hölder's inequality on the bound for the fourth moment, proved by Heath-Brown \cite[Theorem 2]{Hea}.

The second is conditional under RH, and it says that for any fixed $\sigma>1/2$, we have \begin{equation}\label{ZetaNearCriticalLine}
\lab\frac1{\zeta(\sigma+it)}\rab\ll_\epsilon(1+|t|)^\epsilon.
\end{equation} It follows from \cite[Theorem 2]{CC}.
\smallskip

For a function $f(x)$, we denote by $\hat f(s)$ its Mellin transform, which is defined as \begin{equation}\label{DefinitionMellinTransform}
\hat f(s)=\int_{0}^\infty f(x)x^{s-1}dx,
\end{equation}
when the integral converges. If $\hat f$ is analytic in the strip $a<\re(s)<b$, then the inverse Mellin transform is given by \begin{equation}\label{InverseMellinTransform}
f(x)=\frac1{2\pi i}\int\limits_{(c)}x^{-s}\hat f(s)ds,
\end{equation}
where the integral is over the vertical line $\re(s)=c$, and $a<c<b$ is arbitrary.
\smallskip

We will use the following estimate for the Gamma function, which is a consequence of Stirling's formula: for a fixed $\sigma\in\R$ and $|t|\geq1$, we have \begin{equation}\label{StirlingEstimate}
|\Gamma(\sigma+it)|\asymp e^{-|t|\frac\pi2}|t|^{\sigma-1/2}.
\end{equation} We will also use the formula \begin{equation}\label{RatioOfGammaFactors}
\frac{\Gamma\bfrac{1-s}{2}}{\Gamma\bfrac s2}=\frac{2^s\sin (\pi s/2)\Gamma(1-s)}{\sqrt\pi}.
\end{equation}

We write the functional equation for the Riemann zeta function as \begin{equation}\label{key}
\zeta(s)=\chi(s)\zeta(1-s),
\end{equation}
where 
\begin{equation}\label{key}
|\chi(\sigma+it)|\ll_{\sigma}(1+|t|)^{1/2-\sigma}.
\end{equation}

We will also use the estimate
\begin{equation}\label{Second moment of zeta}
\int_{-T}^T|\zeta(\sigma+it)|^2dt\ll T^{1+\epsilon},
\end{equation}
which is true for any $\sigma\geq1/2.$

\msection{Outline of the proof and double Dirichlet series}\label{SectionMultipleDirichletSeries}

Applying Mellin inversion to $S(X,Y;\phi,\psi)$ twice, we obtain \begin{equation}\label{ApplyingMellinInversion}
\sw=\bfrac1{2\pi i}^2\int\limits_{(\sigma)}\int\limits_{(\omega)}A(s,w)X^wY^s\hat\phi(w)\hat\psi(s)dwds,
\end{equation}
where for $\re(s)=\sigma$ and $\re(w)=\omega$ large enough, we have the absolutely convergent double Dirichlet series \begin{equation}\label{DoubleDirichletSeries}
A(s,w)=\sum_{m\odd}\sum_{n\odd}\frac{\leg mn}{m^wn^s}=\sum_{m\odd}\frac{L_2(s,\chi_m)}{m^w}.
\end{equation} We use the results of Blomer to meromorphically continue $A(s,w)$ to the whole $\C^2$, shift the two integrals to the left and compute the contribution of the crossed polar lines. 

We now cite and sketch the proof of Lemma 2 in \cite{Blo}. For two characters $\psi,\psi'$ of conductor dividing $8$, we define \begin{equation}\label{ZetWithCharacters}
Z(s,w;\psi,\psi'):=\zeta_2(2s+2w-1)\sum_{m,n\odd}\frac{\chi_m(n)\psi(n)\psi'(m)}{m^wn^s},
\end{equation} which converges absolutely if $\re(s)$ and $\re(w)$ are large enough, and we let \begin{equation}\label{DefinitionOfZet}
Z(s,w):=Z(s,w;\psi_1,\psi_1)=\zeta_2(2s+2w-1)A(s,w).
\end{equation}
We also denote $${\bf Z}(s,w;\psi)=\left(
\begin{array}{c}
Z(s,w;\psi,\psi_1)\\
Z(s,w;\psi,\psi_{-1})\\
Z(s,w;\psi,\psi_{2})\\
Z(s,w;\psi,\psi_{-2})\\
\end{array}
\right),\hspace{10pt} {\bf Z}(s,w)=\left(
\begin{array}{c}
{\bf Z}(s,w,\psi_1)\\
{\bf Z}(s,w,\psi_{-1})\\
{\bf Z}(s,w,\psi_{2})\\
{\bf Z}(s,w,\psi_{-2})\\
\end{array}
\right).$$

\begin{theorem}\label{ExtensionOfZ(s,w)}
	The functions $Z(s,w;\psi,\psi')$ have a meromorphic continuation to the whole $\C^2$ with a polar line $s+w=3/2$. There is an additional polar line at $s=1$ with residue $\res_{(1,w)}Z(s,w)=\zeta_2(2w)/2$ if and only if $\psi=\psi_1$, and an additional polar line $w=1$ with residue $\res_{(s,1)}Z(s,w)=\zeta_2(2s)/2$ if and only if $\psi'=\psi_1$. 
	
	The functions $(s-1)(w-1)(s+w-3/2)Z(s,w;\psi,\psi')$ are polynomially bounded in vertical strips, meaning that for fixed $\re(s)$ and $\re(w),$ $(s-1)(w-1)(s+w-3/2)Z(s,w;\psi,\psi')$ is bounded by a polynomial in $\im(s),\im(w)$. The functions satisfy functional equations relating ${\bf Z}(s,w)$ with ${\bf Z}(w,s)$, and ${\bf Z}(s,w)$ with ${\bf Z}(1-s,s+w-1/2)$.
\end{theorem}
\noindent{\bf Remarks:}\begin{enumerate}\item[(i)] Blomer gives explicit $16\times16$ matrices $A$ and $B(s)$, such that ${\bf Z}(s,w)=A\cdot{\bf Z}(w,s)$, and ${\bf Z}(s,w)=B(s)\cdot{\bf Z}(1-s,s+w-1/2)$, we will use the explicit form in \eqref{ExplicitFunctionalEquationForZ} to compute the residues on the polar line $s+w=3/2$. 
	
	\item[(ii)] We can also iterate the two functional equations and obtain others, for example relating ${\bf Z}(s,w)$ with ${\bf Z}(1-s,1-w)$. Blomer also gives an almost explicit form of this case.
	\item[(iii)] For us, a polar line means that if we fix one of the variables, the resulting function of the other variable has a pole on the corresponding line with the given residue. What we state doesn't exactly hold at the points $(1/2,1)$ and $(1,1/2)$, where two of the polar lines intersect, but we will not need to know the exact behavior at these points.
\end{enumerate}

\begin{proof}[Proof sketch]
	We write the Dirichlet series for $Z(s,w;\psi,\psi')$ in two ways.
	
	First, writing $m=m_0m_1^2$ with $\mu^2(m_0)=1$, we have \begin{equation}\label{SievedZ}
	\begin{aligned}
	Z(s,w;\psi,\psi')&=\zeta_2(2s+2w-1)\sum_{\substack{m_0\odd,\\\mu^2(m_0)=1}}\frac{L_2(s,\chi_{m_0}\psi)\psi'(m_0)}{m_0^w}\times\\[5pt]&\hspace{30pt}\times\sum_{m_1\odd}\frac1{m_1^{2w}}\prod_{p|m_1}\lz1-\frac{\chi_{m_0}\psi(p)}{p^s}\pz\\&=\zeta_2(2s+2w-1)\sum_{\substack{m_0\odd,\\\mu^2(m_0)=1}}\frac{L_2(s,\chi_{m_0}\psi)\psi'(m_0)\zeta_2(2w)}{m_0^wL_2(s+2w,\chi_{m_0}\psi)}.
	\end{aligned}
	\end{equation}
	
	If $\psi$ is non-trivial, the right-hand side converges absolutely in the region $$\{(s,w):\re(w)>1\text{ and }\re(s+w)>3/2\},$$ the second condition comes from using the functional equation in the numerator when $\re(s)<1/2$. When $\psi$ is the trivial character, the summand corresponding to $m_0=1$ is $\frac{\zeta_2(s)\zeta_2(2w)}{\zeta_2(s+2w)}$, so there is a pole at $s=1$ with residue $\zeta_2(2w)/2.$ Note that the other potential polar lines coming from $\frac{\zeta_2(2w)}{L_2(s+2w,\chi_{d_0}\psi)}$ are outside of the considered region.
	
	The second way to write $Z(s,w;\psi,\psi')$ is by exchanging summations and using the quadratic reciprocity. We obtain \begin{equation}\label{ZExchangeSum}\begin{aligned}
	Z(s,w;\psi,\psi')&=\zeta_2(2s+2w-1)\sum_{m,n\odd}\frac{\chi_m(n)\psi(n)\psi'(m)}{m^wn^s}\\&=\zeta_2(2s+2w-1)\sum_{n\odd}\frac{L_2(w,\tilde\chi_n\psi')\psi(n)}{n^s}.
	\end{aligned}	
	\end{equation}
	We can again write $n=n_0n_1^2$ with $\mu^2(n_0)=1$ and obtain a series that is absolutely convergent in the region $$\{(s,w):\re(s)>1\text{ and }\re(s+w)>3/2\},$$ unless $\psi'$ is the trivial character, in which case there is a pole at $w=1$ coming from the summands when $n$ is a square, and the residue is $\zeta_2(2s)/2.$
	
	Note that \eqref{ZExchangeSum} gives a link between $Z(s,w;\psi,\psi')$ and $Z(w,s;\psi',\psi)$, which gives us a functional equation relating ${\bf Z}(s,w)$ with ${\bf Z}(w,s)$. To finish the proof and obtain the meromorphic continuation to the whole $\C^2$, we use the functional equation in the numerator of \eqref{SievedZ}, which gives a functional equation relating ${\bf Z}(s,w)$ and ${\bf Z}(1-s,s+w-1/2)$, where the change in the second coordinate comes from the conductor in the functional equation for $L(s,\chi_m)$. Notice that this change of variables interchanges $2s+2w-1$ and $2w$, leaves $s+2w$ fixed, and maps the line $w=1$ to $s+w=3/2$, which becomes a new polar line. 
	
	We can iterate the two transformations coming from \eqref{SievedZ} and \eqref{ZExchangeSum} and obtain a function meromorphic on a tube region of the form~$\{(s,w):\re(s)^2+\re(w)^2>c\}$~for some $c$. 	During this process, we obtain some additional potential polar lines, but these will be canceled by the gamma factors coming from the functional equations.  To obtain a continuation to the region $\{(s,w):\re(s)^2+\re(w)^2\leq c\}$, we use Bochner's Tube theorem from multivariable complex analysis, which states that a function that is holomorphic on a tube region can be continued to its convex hull (see \cite{Boc}).
	
	The proof that the function is polynomially bounded in vertical strips is similar to the proof of Proposition 4.11 in \cite{DGH}.
\end{proof}

We will also use the following estimate, which is Theorem 2 in \cite{Blo}.

\begin{theorem}
	For any $Y_1,Y_2\geq1$ and characters $\psi,\psi'$ modulo 8, we have \begin{equation}\label{Second moment of Z(s,w)}
	\int_{-Y_1}^{Y_1}\int_{-Y_2}^{Y_2}|Z(1/2+it,1/2+iu;\psi,\psi')|^2dudt\ll(Y_1Y_2)^{1+\epsilon}.
	\end{equation}
\end{theorem}

In the next two sections, we are going to shift the two integrals in \eqref{ApplyingMellinInversion} to the left and compute the contribution of the crossed polar lines. By Theorem~\ref{ExtensionOfZ(s,w)}, the polar lines of $A(s,w)=\frac{Z(s,w)}{\zeta_2(2s+2w-1)}$ are the following:
\begin{itemize}
	\item The polar lines of $Z(s,w)$, which give us the main term in Theorem~\ref{MainTheorem}:
	\begin{itemize} 
		\item the line $s=1$ with residue $\res_{(1,w)}A(s,w)=\frac{\zeta_2(2w)}{2\zeta_2(2w+1)}$, 
		\item the line $w=1$ with residue $\res_{(s,1)}A(s,w)=\frac{\zeta_2(2s)}{2\zeta_2(2s+1)}$, 
		\item the line $s+w=3/2$, whose residue will be computed in Lemma~\ref{LemmaResidue3/2}.
	\end{itemize}
	\item Zeros of $\zeta_2(2s+2w-1)=\zeta(2s+2w-1)\lz1-2^{1-2s-2w}\pz$, which are
	the lines $s+w=\frac{\rho+1}{2},$ where $\rho$ is such that $\zeta(\rho)=0$, or $s+w=\frac{k\pi i}{\log 2}+\frac12$ for some $k\in\Z$. All these satisfy $\re(s+w)<1$, and even $\re(s+w)\leq \frac34$ if we assume RH.
\end{itemize}

We will see that the main term comes from the polar lines of $Z(s,w)$, while the polar lines coming from the zeros of $\zeta_2(2s+2w-1)$ determine how far to the left we will be able to shift the integrals, so they give us our error term.

\msection{Contribution of the polar lines $s=1$ and $w=1$}

In this section, we shift the integrals to the left and compute the contribution of the polar lines $s=1$ and $w=1$. We begin with $\sigma=2$ and $\omega=2$ in \eqref{ApplyingMellinInversion}, where everything converges absolutely. 

Then we move the inner integral to the line $\re(w)=3/4+\epsilon$, so we obtain \begin{equation}\label{ShiftingToOmega'}\begin{aligned}
\sw=\lz \frac1{2\pi i}\pz^2&\int\limits_{(2)}\int\limits_{(3/4+\epsilon)}A(s,w)X^wY^s\hat \phi(w)\hat\psi(s)dwds+\\&+\frac1{2\pi i}\int\limits_{(2)}XY^s\hat\phi(1)\hat\psi(s)\res_{(s,1)}A(s,w)ds.
\end{aligned}\end{equation}
This shift of integrals is justified by the fast decay of the Mellin transform and polynomial boundedness of $A(s,w)$ in vertical strips.

Now we compute the second integral in \eqref{ShiftingToOmega'}, which equals \begin{equation}
\frac{\hat\phi(1)X}{2\pi i}\int\limits_{(2)}\frac{Y^s\hat\psi(s)\zeta_2(2s)}{2\zeta_2(2s+1)}ds.
\end{equation}

We again estimate this integral using the residue theorem. The integrand has the following poles:\begin{itemize}
	\item At $s=1/2$ with residue $$\frac{Y^{1/2}\hat\psi\bfrac12}{8\zeta_2(2)}=\frac{Y^{1/2}\hat\psi\bfrac12}{\pi^2}.$$
	\item Zeros of $$\zeta_2(2s+1)=\lz1-\frac{1}{2^{2s+1}}\pz\zeta(2s+1).$$ These are at the points $s=\frac{\rho-1}{2},$ where $\zeta(\rho)=0$, and $$s=\frac{k\pi i}{\log 2}-\frac12,\ k\in\Z.$$ These poles have $\re(s)<0$ and if we assume RH, they all have $\re(s)\leq-1/4$.
\end{itemize}
Therefore, we have the following:\begin{equation}\label{ComputingContributionOfFirstResidue}
\begin{aligned}
\frac{\hat\phi(1)X}{2\pi i}&\int\limits_{(2)}\frac{Y^s\hat\psi(s)\zeta_2(2s)}{2\zeta_2(2s+1)}ds=\frac{\hat\phi(1)\hat\psi\bfrac12XY^{1/2}}{\pi^2}+\frac{\hat\phi(1)X}{2\pi i}\int\limits_{(\delta)}\frac{Y^s\hat\psi(s)\zeta_2(2s)}{2\zeta_2(2s+1)}ds.
\end{aligned}
\end{equation}

Depending whether we assume RH or not, we take $\delta=-\frac14+\epsilon$ or $\delta=\epsilon$, bound the integral trivially (we use $\eqref{ZetaNearCriticalLine}$ when $\delta=-1/4+\epsilon$) and get \begin{equation}
\frac{\hat\phi(1)X}{2\pi i}\int\limits_{(2)}\frac{Y^s\hat\psi(s)\zeta_2(2s)}{2\zeta_2(2s+1)}ds=\frac{\hat\phi(1)\hat\psi\bfrac 12XY^{1/2}}{\pi^2}+O\lz XY^{\delta}\pz.
\end{equation}

Using this in \eqref{ShiftingToOmega'}, we obtain \begin{equation}
\begin{aligned}
\sw=&\frac{\hat\phi(1)\hat\psi\bfrac12XY^{1/2}}{\pi^2}+\\&+\lz\frac{1}{2\pi i}\pz^2\int\limits_{(2)}\int\limits_{(3/4+\epsilon)}A(s,w)X^wY^s\hat\phi(w)\hat\psi(s)dwds+O(XY^{\delta}),
\end{aligned}
\end{equation}

Note that when $\phi=\psi=\unitint$, the first term is $\frac{2}{\pi^{2}}XY^{1/2}$ and corresponds to the contribution when $n$ is a square.

\smallskip

Next, we exchange the integrals and shift the integral over $\re(s)=2$ to $\re(s)=3/4$, crossing the polar line at $s=1$. The computation of the residues coming from this polar line is completely analogous to the previous case, and the result is stated in the following theorem: \begin{theorem}
	Let $\epsilon>0$. Then we have: \begin{equation}\label{TheoremAfter2Shifts}\begin{aligned}
	\sum_{m\odd}\sum_{n\odd}\leg mn\phi\bfrac mX\psi\bfrac nY&=\frac{\hat\phi(1)\hat\psi\bfrac12XY^{1/2}+\hat\psi(1)\hat\phi\bfrac12YX^{1/2}}{\pi^2}+\\&\hspace{-100pt}+\lz\frac{1}{2\pi i}\pz^2\int\limits_{(3/4)}\int\limits_{(3/4+\epsilon)}A(s,w)X^wY^s\hat\phi(w)\hat\psi(s)dwds+O_\epsilon\lz YX^\delta+XY^\delta\pz,
	\end{aligned}\end{equation} where $\delta=\epsilon$. If we assume the Riemann Hypothesis, then we can take $\delta=-1/4+\epsilon$.
\end{theorem}

\msection{Contribution of the polar line $s+w=3/2$}
Before further shifting the integrals, we need to compute the residues on the polar line $s+w=3/2$, which is done in the following lemma.
\begin{lemma}\label{LemmaResidue3/2}
	For all $s\in\C$, \begin{equation}\label{ResidueOns+w=3/2}
	\res_{\lz s,\frac32-s\pz}Z(s,w)=\frac{\sqrt\pi\sin\bfrac{\pi s}{2}\Gamma\lz s-\frac12\pz\zeta(2s-1)}{2(2\pi)^s}.
	\end{equation}
\end{lemma}
\begin{proof}
	We use the functional equation (28) in \cite{Blo}, from which it follows that \begin{equation}\label{ExplicitFunctionalEquationForZ}\begin{aligned}
	&Z(1-u,u+v-1/2)=\frac{\pi^{-u+\frac12}\Gamma\lz\frac u2\pz}{\lz 4^{1-u}-4\pz\Gamma\lz\frac{1-u}2\pz}\cdot\Bigg(-4^uZ(u,v;\psi_1,\psi_1)+\\&+\lz 4^u-2\pz Z(u,v;\psi_1,\psi_{-1})+\lz 2^u-2^{1-u}\pz \lz Z(u,v;\psi_1,\psi_2)+ Z(u,v;\psi_1,\psi_{-2})\pz\Bigg).
	\end{aligned}\end{equation}
	Under the change of variables $(s,w)=(1-u,u+v-1/2)$, the line $v=1$ transforms to the line $s+w=3/2$. Since $v=1$ is a polar line of $Z(u,v;\psi,\psi')$ if and only if $\psi'=\psi_1$, the residue comes only from the first term in the parenthesis on the right-hand side of \eqref{ExplicitFunctionalEquationForZ}, and is given by $$\res_{\lz1-u,u+\frac12\pz}Z(u,v;\psi_1,\psi_1)=\frac{\pi^{-u+\frac12}\Gamma\bfrac u2\lz-4^u\pz\zeta_2(2u)}{2\lz4^{1-u}-4\pz\Gamma\bfrac{1-u}{2}}=\frac{\pi^{-u+\frac12}\Gamma\bfrac u2\zeta(2u)}{2\cdot4^{1-u}\Gamma\bfrac{1-u}{2}},$$ so we have $$
	\begin{aligned}
	\res_{\lz s,\frac32-s\pz}Z(s,w)&=\frac{\pi^{s-\frac12}\Gamma\lz\frac {1-s}2\pz\zeta(2-2s)}{2\cdot4^s\Gamma\lz\frac{s}{2}\pz}=\frac{\sqrt\pi\sin\bfrac{\pi s}{2}\Gamma\lz s-\frac12\pz\zeta(2s-1)}{2(2\pi)^s},
	\end{aligned}$$ where the last equality follows after using the functional equation for $\zeta(s)$ and the formula \eqref{RatioOfGammaFactors}.
\end{proof}
We are now ready to prove Theorem~\ref{MainTheoremSmoothed}:
\begin{proof}[Proof of Theorem \ref{MainTheoremSmoothed}]
	We move the integral from \eqref{TheoremAfter2Shifts} further to the left. According to the discussion at the end of Section \ref{SectionMultipleDirichletSeries}, we know that except the line $s+w=3/2$, the integrand has no poles with $\re(s+w)\geq1$, or with $\re(s+w)>3/4$ if we assume RH. Hence we have \begin{equation}\label{ShiftingPast3/2}\begin{aligned}
	\lz\frac{1}{2\pi i}\pz^2&\int\limits_{(3/4)}\int\limits_{(3/4+\epsilon)}A(s,w)X^wY^s\hat\phi(w)\hat\psi(s)dwds\\[5pt]&=
	\lz\frac{1}{2\pi i}\pz^2\int\limits_{(3/4)}\int\limits_{(\delta')}A(s,w)X^wY^s\hat\phi(w)\hat\psi(s)dwds+\\[5pt]&+\frac1{2\pi i}\int\limits_{(3/4)}X^{\frac32-s}Y^{s}\hat\phi\lz\frac32-s\pz\hat\psi\lz s\pz\res_{\lz s,\frac32-s\pz}A(s,w)ds,
	\end{aligned}\end{equation} where $\delta'=1/4+\epsilon$, or $\epsilon$ under RH. Using Lemma \ref{LemmaResidue3/2} and \eqref{DefinitionOfZet}, we have $$\res_{\lz s,\frac32-s\pz}A(s,w)=\frac{\res_{\lz s,\frac32-s\pz}Z(s,w)}{\zeta_2(2)}=\frac{\sqrt\pi\sin\bfrac{\pi s}{2}\Gamma\lz s-\frac12\pz\zeta(2s-1)}{2\zeta_2(2)(2\pi)^s}.$$ Therefore the second integral on the right-hand side of \eqref{ShiftingPast3/2} is \begin{equation}
	\begin{aligned}
	\frac{2X^{\frac 32}}{i\pi^{\frac52}}\int\limits_{(3/4)}\bfrac{Y}{2\pi X}^s\hat\phi\lz\frac32-s\pz\hat\psi(s)\Gamma\lz s-\frac12\pz\sin\bfrac{\pi s}2\zeta(2s-1)ds.
	\end{aligned}
	\end{equation}
	Hence we have \begin{equation}\label{FinalComputationSmoothedSum}
	\begin{aligned}
	\sw&=\frac{\hat\phi(1)\hat\psi\bfrac12XY^{1/2}+\hat\psi(1)\hat\phi\bfrac12YX^{1/2}}{\pi^2}+\\&\hspace{-20pt}+\frac{2X^{\frac32}}{i\pi^{\frac52}}\int\limits_{(3/4)}\bfrac{Y}{2\pi X}^s\hat\phi\lz\frac32-s\pz\hat\psi(s)\Gamma\lz s-\frac12\pz\sin\bfrac{\pi s}2\zeta(2s-1)ds+\\&\hspace{-20pt}+\bfrac1{2\pi i}^2\int\limits_{(3/4)}\int\limits_{(\delta')}A(s,w)X^wY^s\hat\phi(w)\hat\psi(s)dwds+O\lz XY^\delta+YX^\delta\pz\\[5pt]&\hspace{-60pt}=\frac2{\pi^2}\cdot X^{3/2}\cdot D\lz\frac YX;\phi,\psi\pz+O\lz XY^\delta+YX^\delta\pz,
	\end{aligned}
	\end{equation}
	where the last equality follows after trivially bounding the second integral, and using $X^{\delta'}Y^{\frac34}\ll XY^\delta+YX^\delta$.
\end{proof}	
\msection{Removing the smooth weights}

In this section, we show how to remove the smooth weights from Theorem~\ref{MainTheoremSmoothed} and prove Theorem \ref{MainTheorem}. We choose the weights $\phi=\psi$ to be a smooth function which is $1$ on the interval $ \left[\frac1U,1-\frac1U\right]$ for some $U$ to be chosen later, and 0 outside of $(0,1)$, and which satisfies \begin{equation}\label{SizeOfMellinTransform}
|\hat\phi(\sigma+it)|\ll_{j,\sigma}\frac{U^{j-1}}{1+|t|^j}
\end{equation} for all $j\geq1$. Then using the P\'olya-Vinogradov inequality (see (3.1) in \cite{CFS}), we have \begin{equation}\label{DifferenceSmoothAndNonsmooth}
|S(X,Y)-S(X,Y;\phi,\psi)|\ll\frac{X^{3/2}+Y^{3/2}}{U}\log(XY).
\end{equation}

Now we need to estimate the dependence on $U$ of the error term in the computation of $S(X,Y;\phi,\psi)$. These come from the following:
\begin{itemize}
	\item The error from the polar lines $s=1$ and $w=1$;
	\item The error from the shifted integral;
	\item The difference of the main terms.
\end{itemize}

The error from the polar lines $s=1$ is
\begin{equation}\label{Error from polar line s=1}
\ll XY^{\sigma}\lab\int_{(\delta)}\frac{\hat\psi(s)\zeta_2(2s)}{\zeta_2(2s+1)}ds\rab,
\end{equation} where $\delta=\epsilon$, or $-1/4+\epsilon$, and similarly for the error from the polar line $w=1$. We have $\frac{1}{\zeta_2(2\delta+2it+1)}\ll1$ in both cases by \eqref{ZetaNearCriticalLine}.

We can shift the integral \eqref{Error from polar line s=1} to any vertical line $\re(s)=\sigma$ with $1/4\geq\sigma\geq\delta$, and bound the integral using the Cauchy-Schwarz inequality as
\begin{equation}\label{key}
\begin{aligned}
&\ll\int_{(\sigma)}|\hat\psi(s)\zeta(2s)|ds\\
&\ll\lz\int_{(\sigma)}\frac{|\zeta(1-2s)|^2}{(1+|s|)^{1+\epsilon}}ds\pz^{\frac12}\lz\int_{(\sigma)}|\hat\psi(s)\chi(2s)|^2(1+|s|)^{1+\epsilon}ds\pz^{\frac12}\\
&\ll U^{j-1}\int_{(\sigma)}(1+|s|)^{2-4\sigma-2j+\epsilon}ds\\
&\ll U^{1/2-2\sigma+\epsilon},
\end{aligned}
\end{equation}
where the first integral on the second line converges by \eqref{Second moment of zeta}, and we took $j=3/2-2\sigma+\epsilon.$

A similar computation for the polar line $w=1$ gives the same result with $X,Y$ interchanged, so the error from these terms is 
\begin{equation}\label{key}
(XY^\sigma+YX^{\sigma})U^{1/2-2\sigma+\epsilon}.
\end{equation}
\medskip\medskip

For the error coming from the shifted integral, we need to estimate 
\begin{equation}\label{Error from shifted double integral}
\lab\int_{(\sigma)}\int_{(\omega)}A(s,w)\hat\phi(w)\hat\psi(s)X^wY^sdwds\rab.
\end{equation}
We have \begin{equation}\label{key}
A(s,w)=\frac{Z(s,w)}{\zeta_2(2s+2w-1)},
\end{equation}
so the integral is \begin{equation}\label{key}
\ll X^\omega Y^\sigma\int_{(\sigma)}\int_{(\omega)}\lab Z(s,w)\hat\phi(w)\hat\psi(s)\rab dwds
\end{equation} provided $\sigma+\omega>1+\epsilon$ or $>3/4+\epsilon$ if we assume RH. We take $\sigma=\omega=1/2+\epsilon$ and estimate the double integral using the Cauchy-Schwarz inequality as follows:
\begin{equation}\label{key}
\begin{aligned}
&\int_{(\frac12+\epsilon)}\int_{(\frac12+\epsilon)}\lab Z(s,w)\hat\phi(w)\hat\psi(s)\rab dwds\\
&\ll\lz\int\limits_{(\frac12+\epsilon)}\int\limits_{(\frac12+\epsilon)}\frac{\lab Z(s,w)\rab^2}{(1+|s|)^{1+\epsilon}(1+|w|)^{1+\epsilon}}dwds\pz^{\frac12}\times\\
&\times\lz\int\limits_{(\frac12+\epsilon)}\int\limits_{(\frac12+\epsilon)}\lab\hat\phi(w)\hat\psi(s)\rab^2(1+|s|)^{1+\epsilon}(1+|w|)^{1+\epsilon} dwds\pz^{\frac12}\\
&\ll U^{j_1+j_2-2}\lz\int\limits_{(\frac12+\epsilon)}\int\limits_{(\frac12+\epsilon)}(1+|s|)^{1+\epsilon-2j_1}(1+|w|)^{1+\epsilon-2j_2} dwds\pz^{1/2}\ll\\
&\ll U^{\epsilon},
\end{aligned}
\end{equation} where the integral on the second line converges by \eqref{Second moment of Z(s,w)}, and we took $j_1=j_2=1+\epsilon.$
It follows that the error from the shifted integral is \begin{equation}\label{key}
\ll (XY)^{1/2+\epsilon}U^{\epsilon}.
\end{equation}
\medskip\medskip

The error from the difference of the main terms is
\begin{equation}\label{Difference of main terms}
\begin{aligned}
E(X,Y;\phi,\psi)&=\frac{2}{\pi^2}X^{3/2}\lz D\lz\frac YX;\phi,\psi\pz-D\bfrac YX\pz \\\
&\hspace{-50pt}\ll XY^{1/2}\lab \hat\phi(1)\hat\psi(1/2)-2\rab+YX^{1/2}\lab\hat \psi(1)\hat\phi(1/2)-2\rab+\\
&\hspace{-50pt}+X^{3/2}\int_{(3/4)}\lab\bfrac{Y}{2\pi X}^s\Gamma\lz s-\frac12\pz\sin\bfrac{\pi s}{2}\zeta(2s-1)\rab\times\\
&\hspace{-50pt}\hspace{35pt}\times\lab\hat\phi(3/2-s)\hat\psi(s)-\frac1{s(3/2-s)}\rab ds\ll\\
&\hspace{-50pt}\ll XY^{1/2}\lab \hat\phi(1)\hat\psi(1/2)-2\rab+YX^{1/2}\lab\hat \psi(1)\hat\phi(1/2)-2\rab+\\
&\hspace{-50pt}+(XY)^{3/4}\int_{(3/4)}\lz1+|t|\pz^{-\frac14}\lab\zeta(2s-1)\rab\lab\hat\phi\lz\frac32-s\pz\hat\psi(s)-\frac1{s\lz\frac32-s\pz}\rab ds.
\end{aligned}
\end{equation}
We estimate the Mellin transforms in the following lemma:
\begin{lemma} For $s=\sigma+it,$ $0<\sigma<1$ and $\phi,\psi$ as above, we have
	
	\begin{equation}\label{Mellin transform estimate 1}
	\hat\phi(1)\hat\psi(1/2)=2+O\lz U^{-1/2}\pz,
	\end{equation} and
	\begin{equation}\label{Mellin transform estimate 2}
	\lab\hat\phi(3/2-s)\hat\psi(s)-\frac1{s(3/2-s)}\rab\ll\frac{\lab\hat\phi(3/2-s)\rab}{U^\sigma}+\frac{1}{|s|U^{3/2-\sigma}}.
	\end{equation}
\end{lemma}
\begin{proof}
	For \eqref{Mellin transform estimate 1}, we have \begin{equation}\label{key}
	\begin{aligned}
	&\hat\phi(1)\hat\psi(1/2)=\int_{0}^\infty \phi(u)du\int_0^\infty\psi(v)v^{-1/2}dv=\\
	&=\lz\lz\int\limits_{0}^{\frac 1U}+\int\limits_{1-\frac 1U}^{1}\pz\phi(u)du+1-\frac2U\pz\lz\lz\int\limits_{0}^{\frac 1U}+\int\limits_{1-\frac1U}^{1}\pz\frac{\psi(v)}{\sqrt v}dv+\int\limits_{\frac1U}^{1-\frac1U}\frac{1}{\sqrt v}dv\pz=\\
	&=\lz1+O\lz\frac 1U\pz\pz\lz O\lz\frac1{\sqrt U}\pz+2\lz\sqrt{1-\frac{1}{U}}-\frac1{\sqrt U}\pz\pz=\\
	&= 2+O\lz U^{-1/2}\pz.
	\end{aligned}
	\end{equation}
	
	For \eqref{Mellin transform estimate 2}, we first use the triangle inequality to get \begin{equation}\label{key}
	\begin{aligned}
	&\lab\hat\phi(3/2-s)\hat\psi(s)-\frac1{s(3/2-s)}\rab\\
	&\leq\lab\hat\phi(3/2-s)\rab\lab\hat\psi(s)-\frac1s\rab+\lab\frac1s\rab\lab\hat\phi(3/2-s)-\frac{1}{3/2-s}\rab.
	\end{aligned}
	\end{equation}
	Now we have \begin{equation}\label{key}
	\begin{aligned}
	\lab\hat\psi(s)-\frac1s\rab&\leq\int_{0}^\infty\lab\psi(u)-1\rab u^{\sigma-1}du=\\
	&=\int_{0}^{\frac1U}\lab\psi(u)-1\rab u^{\sigma-1}du+\int_{1-\frac1U}^{1}\lab\psi(u)-1\rab u^{\sigma-1}du\ll\\
	&\ll \frac1{U^\sigma},
	\end{aligned}
	\end{equation}
	and similarly \begin{equation}\label{key}
	\lab\hat\phi(3/2-s)-\frac1{3/2-s}\rab\leq\int_{0}^\infty\lab\phi(u)-1\rab u^{1/2-\sigma}du\ll\frac{1}{U^{3/2-\sigma}}.
	\end{equation}
\end{proof}
Using this Lemma in \eqref{Difference of main terms}, we get
\begin{equation}\label{key}
E(X,Y;\phi,\psi)\ll\frac{XY^{1/2}+YX^{1/2}}{\sqrt U}+\bfrac{XY}{U}^{3/4}.
\end{equation}
\medskip\medskip

Putting everything together, the error in both cases is \begin{equation}\label{key}
\frac{X^{\frac32+\epsilon}+Y^{\frac32+\epsilon}}{U}+\frac{XY^{\frac12}+YX^{\frac12}}{\sqrt U}+\bfrac{XY}{U}^{\frac34}+(XY^{\sigma}+YX^{\sigma})U^{\frac12-2\sigma+\epsilon}+(XY)^{\frac12}U^{\epsilon},
\end{equation} where $1/4\geq\sigma\geq\delta$. In both cases, the best choice is $\sigma=1/4$. Then we can take $U=\min\{X,Y\}$ and obtain the error  \begin{equation}\label{key}
XY^{1/4+\epsilon}+YX^{1/4+\epsilon}
\end{equation} in the range $Y^{2/5}\leq X\leq Y^{5/2}.$
In the remaining range, the result follows from \eqref{PolyaVinogradov} and the asymptotic expansion of $C(\alpha)$ \eqref{AsymptoticEstimateForC1} and \eqref{AsymptoticEstimateForC2}, together with the fact that $C(\alpha)=D(\alpha)$ as will be proved in the next section.

\msection{Proving that $C(\alpha)=D(\alpha)$.}\label{SectionComparingMainTerms}
In this section, we show that $C(\alpha)=D(\alpha)$. Recall that \begin{equation}\label{D}
\begin{aligned}
D(\alpha)=\sqrt\alpha+\alpha-\frac{2\sqrt\pi}{2\pi i}\int\limits_{(3/4)}\bfrac{\alpha}{2\pi}^s\cdot\frac{\Gamma\lz s-\frac32\pz\sin\bfrac{\pi s}{2}\zeta(2s-1)}{s}ds.
\end{aligned}
\end{equation}
We shift the integral to the left, capturing the pole at $s=\frac12$, which contributes $$2\sqrt\pi\bfrac{\alpha}{2\pi}^{1/2}\cdot\frac{-\sin(\pi/4)\zeta(0)}{1/2}=\alpha^{1/2}.$$ The horizontal integrals vanish by \eqref{StirlingEstimate} and a convexity estimate for $\zeta(s)$, so \begin{equation}\label{DShifted}
\begin{aligned}
D(\alpha)&=\alpha-\frac{2\sqrt\pi}{2\pi i}\int\limits_{(-1/4)}\bfrac{\alpha}{2\pi}^s\cdot\frac{\Gamma\lz s-\frac32\pz\sin\bfrac{\pi s}{2}\zeta(2s-1)}{s}ds\\&=\alpha-\frac{2\sqrt\pi}{2\pi i}\int\limits_{(1/4)}\alpha^{-s}\cdot\frac{(2\pi)^s\Gamma\lz-s-\frac32\pz\sin\bfrac{\pi s}{2}\zeta(-2s-1)}{s}ds.
\end{aligned}
\end{equation}
This integral is an inverse Mellin transform, so we rewrite $C(\alpha)$ using Mellin inversion.
We have \begin{equation}\label{C}
\begin{aligned}
C(\alpha)&=\alpha+\alpha^{3/2}\frac{2}{\pi}\sum_{k=1}^\infty\frac{1}{k^2}\int_{0}^{1/\alpha}\sqrt y\sin\bfrac{\pi k^2}{2y}dy\\&=\alpha+\frac2\pi\sum_{k=1}^\infty\frac1{k^2}\int_0^1\sqrt u\sin\bfrac{\pi k^2\alpha}{2u}du\\&=\alpha+\frac2{\pi}\cdot\frac{1}{2\pi i}\int\limits_{(c)}\alpha^{-s}\hat f(s)ds,
\end{aligned}
\end{equation}
where \begin{equation}\label{f}
f(x)=\sum_{k=1}^\infty\frac1{k^2}\int_0^1\sqrt u\sin\bfrac{\pi k^2 x}{2u}du.
\end{equation}
For $0<\re(s)<1$, we have \begin{equation}\label{MellinTransformOff1}
\begin{aligned}
\hat f(s)&=\int_0^\infty\sum_{k=1}^\infty\frac1{k^2}\int_0^1\sqrt u\sin\bfrac{\pi k^2 x}{2u}du\ x^{s-1}dx\\&=\sum_{k=1}^\infty\frac1{k^2}\int_0^1\sqrt u\int_0^\infty \sin\bfrac{\pi k^2 x}{2u} x^{s-1}dx\ du,\end{aligned}
\end{equation}
which isn't obvious as the double integral doesn't converge absolutely, but we will justify the interchange of summation and integrals in Lemma \ref{LemmaInterchangeOfOperations}. We can now make a change of variables $y=\frac{\pi k^2x}{2u}$ and obtain
\begin{equation}\label{MellinTransformOff2}
\begin{aligned}
\hat f(s)&=\sum_{k=1}^\infty\frac1{k^2}\int_0^1\sqrt u\int_0^\infty \sin(y) y^{s-1}dy\ \bfrac{2u}{\pi k^2}^sdu\\&=\bfrac{2}{\pi}^s\zeta(2+2s)\int_0^1u^{s+1/2}du\int_0^\infty\sin(y)y^{s-1}dy\\&=\frac{2^s\zeta(2+2s)\Gamma(s)\sin\bfrac{\pi s}{2}}{\pi^s(s+3/2)},
\end{aligned}
\end{equation}
which holds for $0<\re(s)<1,$ so we can take $c=1/4$ in \eqref{C}.
It therefore suffices to show that \begin{equation}\label{ToShow}
-2\sqrt\pi\frac{(2\pi)^s\Gamma\lz-s-\frac32\pz\sin\bfrac{\pi s}{2}\zeta(-2s-1)}{s}=\frac2\pi\cdot\frac{2^s\zeta(2+2s)\Gamma(s)\sin\bfrac{\pi s}{2}}{\pi^s(s+3/2)}.
\end{equation}
The functional equation for the zeta function gives \begin{equation}\label{}
\zeta(2s+2)=\pi^{2s+3/2}\cdot\frac{\Gamma\lz-s-\frac12\pz}{\Gamma(s+1)}\zeta(-2s-1),
\end{equation}
so using $s\Gamma(s)=\Gamma(s+1)$ and $(-s-3/2)\Gamma(-s-3/2)=\Gamma(-s-1/2)$ gives the result.

It remains to justify the interchange of the order of summation and integrations in \eqref{MellinTransformOff1}:
\begin{lemma}\label{LemmaInterchangeOfOperations}
	If $0<\re(s)<1$, it holds that \begin{equation}\label{key}
	\begin{aligned}
	&\int_0^\infty\sum_{k=1}^\infty\frac1{k^2}\int_0^1\sqrt u\sin\bfrac{\pi k^2 x}{2u}du\ x^{s-1}dx\\&=\sum_{k=1}^\infty\frac1{k^2}\int_0^1\sqrt u\int_0^\infty \sin\bfrac{\pi k^2 x}{2u} x^{s-1}dx\ du.
	\end{aligned}
	\end{equation}
\end{lemma}
\begin{proof}
	We have \begin{equation}\label{key}
	\begin{aligned}
	&\int_0^\infty\sum_{k=1}^\infty\frac1{k^2}\int_0^1\sqrt u\sin\bfrac{\pi k^2 x}{2u}du\ x^{s-1}dx\\&=\lim_{A\rightarrow\infty}\int_{1/A}^A\sum_{k=1}^\infty\frac1{k^2}\int_0^1\sqrt u\sin\bfrac{\pi k^2 x}{2u}du\ x^{s-1}dx.
	\end{aligned}
	\end{equation}
	We can now interchange the integrals and summation, because \begin{equation}\label{key}
	\begin{aligned}
	\sum_{k=1}^\infty\frac1{k^2}\int_0^1\int_{1/A}^A\lab\sqrt u\sin\bfrac{\pi k^2 x}{2u}x^{s-1}\rab dx\ du\ll A,
	\end{aligned}
	\end{equation}
	so \begin{equation}\label{key}
	\begin{aligned}
	&\lim_{A\rightarrow\infty}\int_{1/A}^A\sum_{k=1}^\infty\frac1{k^2}\int_0^1\sqrt u\sin\bfrac{\pi k^2 x}{2u}du\ x^{s-1}dx\\&=\lim_{A\rightarrow\infty}\sum_{k=1}^\infty\frac1{k^2}\int_0^1\sqrt u\int_{1/A}^A\sin\bfrac{\pi k^2 x}{2u}x^{s-1}dx\ du.
	\end{aligned}
	\end{equation}
	To insert the limit inside the sum and integral, we use the Dominated convergence theorem with the bound $\sqrt u|\int_{1/A}^A\sin\bfrac{\pi k^2 x}{2u}x^{s-1}dx|\leq K\sqrt u$ for an absolute constant $K$ independent of $u$ and $A$ in our range, which holds because $\int_0^\infty\sin(y)y^{s-1}dy$ converges.
\end{proof}

\end{document}